\documentclass[10pt]{article}
\usepackage{amsmath,amssymb,wasysym}
\usepackage{graphicx}
\usepackage{color}
\usepackage[notcite,notref]{}
\usepackage[left=3cm,right=3cm, top=3.3cm, bottom=3cm]{geometry}
\newtheorem{theo}{Theorem}[section]
\newtheorem{lemm}[theo]{Lemma}

\newtheorem{prop}[theo]{Proposition}

\newtheorem{remark}[theo]{Remark}
\def\proof {{\noindent \bf{Proof:\hspace{4pt}}}}
\def\endproof{\hfill$\square$\vspace{6pt}}
\numberwithin{equation}{section}
\title{
{\bf\Large  Existence theory for the Boussinesq equation in Modulation spaces}}
\author{{Carlos Banquet}{\thanks{
Corresponding author.}} \\
{\small Departamento de Matem\'{a}ticas y Estad\'{\i}stica, Universidad de C%
\'{o}rdoba}\\
{\small A.A. 354, Monter\'{\i}a, Colombia.}\\
{\small \texttt{E-mail:cbanquet@correo.unicordoba.edu.co}}\vspace{.5cm}\\
{{\'Elder J. Villamizar-Roa}}\\
{\small Universidad Industrial de Santander, Escuela de Matem\'{a}ticas}\\
{\small A.A. 678, Bucaramanga, Colombia.} \\
{\small \texttt{E-mail:jvillami@uis.edu.co}}}
\date{}
\begin{document}
\maketitle
\begin{abstract}
In this paper we study the Cauchy problem for the generalized Boussinesq equation with initial data in modulation spaces $M^{s}_{p^\prime,q}(\mathbb{R}^n),$ $n\geq 1.$ After a decomposition of the Boussinesq equation in a $2\times 2$-nonlinear system, we obtain the existence of global and local solutions in several classes of functions with values in $ M^s_{p,q}\times D^{-1}JM^s_{p,q}$ spaces for suitable $p,q$ and $s,$ including the special case $p=2,q=1$ and $s=0.$ Finally, we prove some results of scattering and asymptotic stability in the framework of modulation spaces.\\

{\bf Keywords.} Boussinesq equation, modulation spaces, local and global solutions, scattering, asymptotic stability.\\

{\bf AMS subject classifications.} 35Q53; 35A01; 47J35; 35B40; 35B35
\end{abstract}

\section{Introduction}
We consider the initial value problem associated to the generalized Boussinesq equation
\begin{equation}\label{GoodBous}
\left\{
\begin{array}{lc}
\partial^2_{t}u-\Delta u+\Delta^2u+\Delta f(u)=0, & (x,t)\in \mathbb{R}^{n+1}, \\
u(x,0)=u_{0}(x), \ \ \ \partial_t u(x,0)=\phi(x)=\Delta v_0(x),&  x\in \mathbb{R}^{n},
\end{array}
\right.
\end{equation}
where $u:\mathbb{R}^n\times \mathbb{R}\rightarrow \mathbb{R}$ is the unknown, $u_0,v_0:\mathbb{R}^n\rightarrow \mathbb{R}$ are given functions denoting the initial data and the nonlinear term is $f(u)=u^\lambda,$ for some $1<\lambda<\infty.$ Equation (\ref{GoodBous}) is physically relevant in the modelling of shallow water waves, ion-sound waves in plasma, the dynamics of stretched string, and another physical phenomena (see Cho and Ozawa \cite{Cho} , Peregrine \cite{Peregrine}). The IVP (\ref{GoodBous}) is formally equivalent to the following system

\begin{equation}\label{EquiSyst}
\left\{
\begin{array}{lc}
\partial_{t}u=\Delta v, & (x,t)\in \mathbb{R}^{n+1}, \\
\partial_{t}v=u-\Delta u-f(u), & (x,t)\in \mathbb{R}^{n+1}, \\
u(x,0)=u_{0}(x), \ \ \ v(x,0)=v_0(x),&  x\in \mathbb{R}^{n}.
\end{array}
\right.
\end{equation}
From Duhamel's principle the Cauchy problem associated to system (\ref{EquiSyst}) is equivalente to the integral equation
\begin{equation}\label{ms5}
[u(t),v(t)]=B(t)[u_0,v_0]-\int_0^tB(t-\tau)[0,f(u(\tau))]d\tau,
\end{equation}
 where $B$ is the solution of the linear problem associated to (\ref{EquiSyst}). More exactly, for initial data $[u_0,v_0]$ and $t\in \mathbb{R},$ we have
\begin{equation}\label{g1}
B(t)[u_0,v_0]=\int_{\mathbb{R}^n}e^{ix\cdot \xi}\left [\begin{array}{cc}
\cos(t|\xi|\langle\xi\rangle) & -|\xi|\langle\xi\rangle^{-1}\sin(t|\xi|\langle\xi\rangle)\\
|\xi|^{-1}\langle\xi\rangle\sin(t|\xi|\langle\xi\rangle) & \cos(t|\xi|\langle\xi\rangle ) \end{array} \right]d\xi,
\end{equation}
where  $\langle\xi\rangle=(1+|\xi|^2)^{1/2}.$ Note that
\[
B(t)[u_0,v_0]=[B_1(t)u_0+B_2(t)v_0, B_3(t)u_0+B_1(t)v_0],
\]
where $B_1(t), B_2(t)$ and $B_3(t)$ are the multiplier operators with symbols $\cos(t|\xi|\langle\xi\rangle)$, $-|\xi|\langle\xi\rangle^{-1}\sin(t|\xi|\langle\xi\rangle)$ and 
$|\xi|^{-1}\langle\xi\rangle\sin(t|\xi|\langle\xi\rangle),$ respectively.\\

 Several authors have analyzed the local and global existence and long time asymptotic behavior of solutions for (\ref{GoodBous}) (cf. \cite{Bona,Cho,Farah,Farah2,Farah3,Lucas1,Linares,Liu,Tsutsumi} and references therein). In particular, in Bona {\it et al} \cite{Bona}, considering the 1D case, the authors decomposed (\ref{GoodBous}) in the following system 
\begin{equation}\label{EquiSystk}
\left\{
\begin{array}{lc}
\partial_{t}u=\partial_x v, \\
\partial_{t}v=\partial_x(u-\partial_{xx}u-f(u))_x, \\
u(x,0)=u_{0}(x), \ \ \ v(x,0)=v_0(x),
\end{array}
\right.
\end{equation}
and analyzed the local well-posedness with initial data $u_0\in H^{s+2}(\mathbb{R}),$ $\phi=(v_0)_x,$ $v_0\in H^{s+1}(\mathbb{R}),$ $s>0,$ and $f$ smooth. Existence results for $f(u)=\vert u\vert^{\lambda-1}u$ and initial data $u_0\in H^{1}(\mathbb{R}),$ $\phi=(v_0)_{xx},$ $v_0\in H^{1}(\mathbb{R}),$ were obtained by Tsutsumi and Matahashi in \cite{Tsutsumi}. For the same nonlinearity $f(u),$ Linares in \cite{Linares} proved the local well-posedness with either $[u_0,\phi]=[u_0,(v_0)_x]\in L^2(\mathbb{R})\times \dot{H}^{-1}(\mathbb{R})$ and $1<\lambda<5,$ or $[u_0,\phi]=[u_0,(v_0)_x]\in H^1(\mathbb{R})\times L^{2}(\mathbb{R})$ and $1<\lambda<\infty.$ The results of \cite{Linares} were extended by Farah \cite{Farah} for the $n$-dimensional case. Some local well-posedness results with initial data $[u_0,\phi]=[u_0,(v_0)_x]\in H^s(\mathbb{R})\times {H}^{s-1}(\mathbb{R})$ and $f(u)=u^2$ were obtained by Farah in \cite{Farah2}. Results of asymptotic behavior have been addressed by Cho and Ozawa \cite{Cho}, Farah \cite{Farah3} and Liu \cite{Liu}, in the framework of $H^s_p$ and $B^s_{p,2}$-spaces for $s>0$ and $1<p<2$. More exactly, in \cite{Liu} the author analyzed the existence of solutions in one-dimensional case for $f(u)=\vert u\vert^{\lambda-1}u,$ $\lambda>5,$ and small initial data $[u_0,v_0]\in (H^1(\mathbb{R})\times L^2(\mathbb{R}))\cap X^1_0,$ where $X^r_s:=L^r_s\times L^r_{s-1},$ being $L^r_s$ the Bessel potential space with potential $J^s=(1-\partial_{xx})^{s/2}.$ In the same paper, a scattering result for small perturbations was obtained. The results of \cite{Liu} were extended in \cite{Cho} to the case $n\geq 1$ in the framework of Besov $B^s_{p,2}$-spaces. The initial data considered in \cite{Cho} belongs to a subset of $B^{s+n\delta}_{\frac{\lambda+1}{\lambda},2}\times \Omega(B^{s+n\delta}_{\frac{\lambda+1}{\lambda},2}),$ $\delta=1-\frac{2}{\lambda+1}$ and $\widehat{\Omega(\psi)}=\vert\xi\vert(\langle \xi\rangle)\hat{\psi}.$ Using the formulation (\ref{EquiSyst}), in \cite{Farah3} was analyzed the reciprocal problem of the scattering theory by constructing a solution with a given scattering state. Existence of local and global solutions for initial data in weak-$L^p$ spaces (also denoted by $ L^{(p,\infty)}$) were obtained by Ferreira in \cite{Lucas1}. Long time behavior and scattering theory results, in the $L^{(p,\infty)}$ framework, were also obtained in \cite{Lucas1}. Previous initial data classes satisfy the following embedding relations
$$B^s_{p,1}\subset H^s_p\subset B^{s}_{p,\infty}\varsubsetneq L^{(q,\infty)},\ \mbox{for}\ s\geq 0,\ \mbox{and}\ \frac{1}{q}=\frac{1}{p}-\frac{s}{n}.$$

Motivated by the previous references, in this paper we study the local and global existence of solutions of (\ref{EquiSyst}) with initial data  in modulation spaces $M_{p,q}^s(\mathbb{R}^n)$. Modulation spaces were introduced by Feichtinger in \cite{Fei}, prompted by the idea of measuring the smoothness classes of functions or distributions.  Since their introduction, modulation spaces have become canonical for both time-frequency and phase-space analysis, see Chaichenets {\it et al} \cite{Chai}. Wang and Hudzik \cite{BaoxHudz} gave an equivalent definition of modulation spaces by using the frequency-uniform-decomposition operators. In the same work, the existence of global solutions for nonlinear Schr\"odinger and Klein-Gordon equations in modulation spaces were analyzed. After then, several studies on nonlinear PDEs in the framework of modulation spaces have been addressed (cf. \cite{Chai,Huang,Iwabuchi,Ruz,Wang} and references therein). In this context, the contribution of this paper is to develop the existence and long time asymptotic behavior of the Boussinesq type equations with initial data in modulation spaces. To get this aim, first we derive some linear estimates for the one parameter group $B$ introduced in (\ref{g1}) (cf. Section \ref{section2}).\\

Before stating our main results, we recall some preliminar definitions and notations related to the modulation spaces $M_{p,q}^s(\mathbb{R}^n)$ (see for instance Wang and Hudzik \cite{BaoxHudz}). Let $Q_k$ be the unit cube with center at $k.$ Let $\rho:\mathbb{R}^n\rightarrow [0,1]$ be a smooth function satisfying $\rho(\xi)=1$ for $\vert \xi\vert_\infty\leq \frac{1}{2}$ and $\rho(\xi)=0$ for $\vert \xi\vert\geq 1.$ Let $\rho_k(\xi)=\rho(\xi-k),$ $k\in\mathbb{Z}^n,$ a translation of $\rho.$ It holds that $\rho_k(\xi)=1$ in $Q_k,$ and thus, $\sum_{k\in\mathbb{Z}^n}\rho_k(\xi)\geq 1$ for all $\xi\in \mathbb{R}^n.$ Let
$$\sigma_k(\xi)=\rho_k(\xi)\left( \sum_{l\in\mathbb{Z}^n}\rho_l(\xi)\right)^{-1},\ k\in\mathbb{Z}^n.$$
Then, the sequence $\{\sigma_k(\xi)\}_{k\in\mathbb{Z}^n}$ verifies the following properties:
\begin{eqnarray*}
&\vert \sigma_k(\xi)\vert\geq C,\ \forall\xi\in Q_k,&\\
&\mbox{supp}(\sigma_k)\subset\{\xi: \vert \xi-k\vert_\infty\leq 1\},&\\
&\sum_{k\in\mathbb{Z}^n}\sigma_k(\xi)=1,\ \forall\xi\in\mathbb{R}^n,&\\
&\vert D^\theta\sigma_k(\xi)\vert\leq C_{\vert \theta\vert},\ \forall\xi\in \mathbb{R}^n,\ \theta\in (\mathbb{Z}^+\cup\{ 0\})^n.&
\end{eqnarray*}
The sequence $\{\sigma_k(\xi)\}_{k\in\mathbb{Z}^n}$ gives a smooth decomposition of $\sigma_k(\xi)=\sigma(\xi-k)$ and 
$$\sigma(\xi)=\rho(\xi)\left( \sum_{l\in\mathbb{Z}^n}\rho_l(\xi)\right)^{-1}.$$
We consider the frequency-uniform decomposition operators $\square_k:=\mathcal{F}^{-1}\sigma_k\mathcal{F},$ $k\in\mathbb{Z}^n.$ Then, for $s\in\mathbb{R},$ $1\leq p,q\leq\infty,$ the modulations spaces $M_{p,q}^s$ are defined as (cf. Wang and Hudzik \cite{BaoxHudz}): 
\begin{eqnarray*}
M_{p,q}^s:=\left\{ f\in \mathcal{S}^{\prime}(\mathbb{R}^n):\ \Vert f\Vert_{M^s_{p,q}}<\infty\right\},
\end{eqnarray*}
where
\begin{equation*}
\Vert f\Vert_{M^s_{p,q}}=\left\{
\begin{array}{lc}
 \left( \sum_{k\in\mathbb{Z}^n}(1+\vert k\vert)^{sq}\Vert \square_kf\Vert_p^q\right)^{1/q},\ \ \mbox{for}\ 1\leq q<\infty,\\
 \\
\sup_{k\in\mathbb{Z}^n}(1+\vert k\vert)^{s}\Vert \square_kf\Vert_p, \ \ \mbox{for}\ q=\infty.
\end{array}
\right.
\end{equation*}
For simplicity, we will write $M^0_{p,q}(\mathbb{R})=M_{p,q}(\mathbb{R}).$ Many of their properties, including embeddings
in other known function spaces, can be found in \cite{BaoxHudz}. In particular, the following embeddings hold:
\begin{itemize}
\item[i)] $M^{s_1}_{p_1,q_1}\subset M^{s_2}_{p_2,q_2},\ \mbox{if}\ s_1\geq s_2,\ 0<p_1\leq p_2,\ 0<q_1\leq q_2,$
\item [ii)] $M^{s_1}_{p,q_1}\subset M^{s_2}_{p,q_2},\ \mbox{if}\ q_1>q_2,\ s_1>s_2,\ s_1-s_2>n/q_2-n/q_1,$
\item [iii)] For $1<p\leq \infty,$ it holds that $M_{p,1}\subset L^\infty\cap L^p,$ 
\item [iv)] For $0<p,q\leq \infty$ and $s\in\mathbb{R},$ it holds that $B^{s+n/q}_{p,q}\subset M^s_{p,q},$
\item [v)] For $0<p,q\leq \infty,$ $s,\sigma\in \mathbb{R},$ the operator $(I-\Delta)^{\sigma/2}:M^{s}_{p,q}\rightarrow M^{s-\sigma}_{p,q}$ is an isomorphic mapping.
\end{itemize}

 Let $D^s=(-\Delta)^{s/2}$ and $J^{s}=(I-\Delta)^{s/2},$ for any $s\in \mathbb{R}.$ Given the Banach space $X=L^\gamma(\mathbb{R};L^p(\mathbb{R}^n))$ or $X=L^\gamma(\mathbb{R};D^{-1}JL^p(\mathbb{R}^n)),$ $1\leq p,\gamma\leq \infty,$ we also consider the function spaces $l_{\square}^q(X),$ $1\leq q<\infty,$ 
introduced in \cite{BaoxHudz}, which are defined as follows:
\begin{align*}
l_{\square}^q(X)=\left\{ f\in \mathcal{S}^{\prime}(\mathbb{R}^{n+1}): \Vert u\Vert_{l_{\square}^q(X)}:= \left( \sum_{k\in\mathbb{Z}^n}\Vert \square_kf\Vert_X^q\right)^{1/q}<\infty\right\}. 
\end{align*}
We consider the following time-dependent spaces in which we will establish our existence results. We denote by $\mathcal{L}^\infty_{\alpha,s} $ the distribution-valued pairs $[u,v]:(-\infty,\infty)\rightarrow M^s_{p,q}\times D^{-1}JM^s_{p,q}$ with norm given by
\begin{align}
& \Vert [u,v]\Vert_{\mathcal{L}^\infty_{\alpha,s}}:=\sup_{-\infty<t<\infty}(1+\vert t \vert)^{\alpha}(\Vert u(t)\Vert_{M_{p,q}^s}+\Vert v(t)\Vert_{D^{-1}JM_{p,q}^s}),
\end{align}
where $\alpha=n\left(\frac 12-\frac 1p\right),$  with $2\leq p,$ and $\Vert v(t)\Vert_{D^{-1}JM_{p,q}^s}:=\Vert J^{-1}Dv(t)\Vert_{M_{p,q}^s}.$ Since $J^{\sigma/2}:M^{s}_{p,q}\rightarrow M^{s-\sigma}_{p,q}$ is an isomorphic mapping, $[u,v]:(-\infty,\infty)\rightarrow M^s_{p,q}\times D^{-1}JM^s_{p,q}$ is equivalent to say that $[u,Dv]:(-\infty,\infty)\rightarrow M^s_{p,q}\times M^{s-1}_{p,q}$ and 
\begin{align}
& \Vert [u,v]\Vert_{\mathcal{L}^\infty_{\alpha,s}}=\sup_{-\infty<t<\infty}(1+\vert t \vert)^{\alpha}(\Vert u(t)\Vert_{M_{p,q}^s}+\Vert Dv(t)\Vert_{M_{p,q}^{s-1}}).
\end{align}

We also consider the space $\mathcal{L}^T_{\alpha,s}$ of the distribution-valued pairs $[u,v]:(-T,T)\rightarrow M^s_{p,q}\times D^{-1}JM^s_{p,q}$ with norm given by
\begin{align}
& \Vert [u,v]\Vert_{\mathcal{L}^T_{\alpha,s}}:=\sup_{-T<t<T}(1+\vert t \vert )^{\alpha}(\Vert u(t)\Vert_{M_{p,q}^s}+\Vert v(t)\Vert_{D^{-1}JM_{p,q}^s}).
\end{align}

Throughout this paper, $\lambda_0(n)$ corresponds to the positive root of the equation $n\lambda^2-(n+2)\lambda-2=0,$ that is, $\lambda_0(n)=\frac{n+2+\sqrt{n^2+12n+4}}{2n}.$ We also denote by $\lambda_1(n)=\frac{n+2}{n-2}$ for $n\geq 3$ and $\lambda_1(n)=\infty$ if $n=1,2.$ Now we are in position to establish the main results of this paper. 
\begin{theo}\label{global1} Let $\lambda$ a positive integer such that $\lambda>\lambda_0(n),$  $p=\lambda+1,$ $1\leq q < 2,$ and $n-\frac{n}{q} \leq s < \frac{n}{q}.$ If $q=1$ assume also $\lambda \geq 2.$ Suppose that $[u_0,v_0]\in M^{s}_{p',q}\times M^{s}_{p',q}.$ There exists $\epsilon>0$ small enough such that if $\tilde{C}(\Vert u_0\Vert_{M^{s}_{p',q}}+\Vert v_0\Vert_{M^{s}_{p',q}})\leq \epsilon,$ the IVP (\ref{EquiSyst}) has a unique global mild solution $[u,v]\in\mathcal{L}^\infty_{\alpha,s},$ satisfying $\Vert [u,v]\Vert_{\mathcal{L}^\infty_{\alpha,s}}\leq 2\epsilon.$ Moreover, the data-solution map $[u_0,v_0]\longmapsto [u,v]$ from $M^{s}_{p',q}\times M^{s}_{p',q}$ into $\mathcal{L}^\infty_{\alpha,s}$ is locally Lipschitz.
\end{theo}
\begin{remark}\label{global2a}
System (\ref{EquiSyst}) has not a scaling relation which makes this
system more awkward than related problems such as the Schr\"odinger or semilinear heat equations
with nonlinearities of type $f(u)=u^\lambda$ (\cite{F-V2,F-V1}). However, as point out in Ferreira \cite{Lucas1}, system (\ref{EquiSyst}) has an intrinsical scaling given by 
\begin{align}\label{sca}
[u,v]\rightarrow [u_\rho, v_\rho]:=\rho^{\frac{2}{\lambda-1}}[u(\rho x,\rho^2 t), v(\rho x,\rho^2 t)], \  \   \mbox{for}\ \rho>0.
\end{align}
Thus, if $\beta=\frac{1-\alpha}{\lambda-1}>0,$ with $\lambda>1,$ $\alpha=n\left(\frac 12-\frac 1p\right)<1,$ $p=\lambda+1,$ then $\beta$ is the unique one such that the norm $ \Vert [u,v]\Vert_{\mathcal{H}^\infty_{\beta,s}}$ defined by 
\begin{align}
 \Vert [u,v]\Vert_{\mathcal{H}^\infty_{\beta,s}}:=\sup_{-\infty<t<\infty}\vert t \vert^{\beta}(\Vert u(t)\Vert_{M_{p,q}^s}+\Vert v(t)\Vert_{D^{-1}JM_{p,q}^s}),
\end{align}
is invariant by (\ref{sca}). Condition $\beta>0$ is equivalent to $\lambda<\lambda_1(n).$ Comparing the norms $\Vert [u,v]\Vert_{\mathcal{H}^\infty_{\beta,s}}$ and $\Vert [u,v]\Vert_{\mathcal{L}^\infty_{\alpha,s}}$ it holds that
 $0<\beta\leq \alpha$ if and only if $\lambda_1(n)>\lambda\geq\lambda_0(n).$ Consequently, it holds that $ \Vert [u,v]\Vert_{\mathcal{H}^\infty_{\beta,s}}\leq \Vert [u,v]\Vert_{\mathcal{L}^\infty_{\alpha,s}}$ and $\mathcal{L}^\infty_{\alpha,s}\subset \mathcal{H}^\infty_{\beta,s},$ for $\lambda_{1}(n)>\lambda\geq\lambda_0(n).$ 
\end{remark}
\begin{remark}\label{global2}
Consider $p,q,s$ as in Theorem \ref{global1} and $1<\lambda<\lambda_1(n).$ There exists $\epsilon>0$ such that if 
\begin{align}
&\sup_{-\infty<t<\infty}\vert t\vert^\beta\Vert B_1(t)u_0\Vert_{M_{p,q}^s}+\sup_{\infty<t<\infty}\vert t\vert^\beta\Vert B_2(t)v_0\Vert_{M_{p,q}^s}<\frac{\epsilon}{2},\label{est14}\\
&\sup_{-\infty<t<\infty}\vert t\vert^\beta\Vert B_3(t)u_0\Vert_{D^{-1}JM_{p,q}^s}+\sup_{-\infty<t<\infty}\vert t\vert ^\beta\Vert B_1(t)v_0\Vert_{D^{-1}JM_{p,q}^s}<\frac{\epsilon}{2},\label{est15}
\end{align}
then the IVP (\ref{EquiSyst}) has a unique global mild solution $[u,v]:(-\infty,\infty)\rightarrow M^s_{p,q}\times D^{-1}JM^s_{p,q}$ satisfying
$\Vert [u,v]\Vert_{\mathcal{H}^\infty_{\beta,s}}<\infty$ (See Remarks \ref{esti_comp}, \ref{rem2} and \ref{rem4} below).
\end{remark}
Theorem \ref{global1} excludes the case $M_{2,1}.$ Next theorem ensures the existence of global solution for initial data in $M_{2,1}\times M_{2,1}.$ 
In this case, the relations $B^{n/2}_{2,1}\subset M_{2,1}\subset L^\infty\cap L^2$ hold. 
\begin{theo}\label{global4}  Let $n\geq 1,$ $p\in [2+\frac{4}{n},\lambda+1]\cap\mathbb{N}$ and $\lambda\in\mathbb{N}$ such that $\lambda>1+\frac{4}{n}.$ Assume that $u_0,v_0 \in M_{2,1}.$ Then, there exists $\epsilon>0$ small enough such that if $\Vert u_0\Vert_{M_{2,1}}\times \Vert v_0\Vert_{M_{2,1}}\leq \epsilon,$ then (\ref{EquiSyst}) has a unique global solution 
\[ [u,v]\in [C(\mathbb{R}; M_{2,1})\cap l_{\square}^1(L^p(\mathbb{R};L^p))]\times  C(\mathbb{R};D^{-1}JM_{2,1}).\]
\end{theo}
Theorem \ref{global1} provides the existence of global solution for $\lambda>\lambda_0(n).$ For $\lambda\leq \lambda_0(n)$ we are able to ensure a local in time solution. This is the content of next theorem.
\begin{theo}\label{local1}  Asume that $\lambda$ is a positive integer such that $1<\lambda\leq \lambda_0(n),$ $p=\lambda+1,$  $1\leq q < \infty,$ $n-\frac{n}{q} \leq s < \frac{n}{q}.$ If $q=1$ assume also that $\lambda\geq 2.$ Then, if $[u_0,v_0]\in M^{s}_{p',q}\times M^{s}_{p',q}$, there exists $0<T<\infty$ such that the IVP (\ref{EquiSyst}) has a unique local mild solution $[u,v]\in\mathcal{L}^T_{\alpha,s}.$ Moreover, the data-solution map $[u_0,v_0]\longmapsto [u,v]$ from $M^{s}_{p',q}\times M^{s}_{p',q}$ into $\mathcal{L}^T_{\alpha,s}$ is locally Lipschitz.

\end{theo}
Now we establish a result on the scattering theory which describes the asymptotic behavior of solutions for the
Boussinesq system in the framework of modulation spaces. We find an initial
data $[u_{0}^{\pm},v_{0}^{\pm}]$ such that the solution $[u^{\pm
},v^{\pm}]$ of the linear problem associated to the system (\ref{EquiSyst}), with initial
data $[u_{0}^{\pm},v_{0}^{\pm}],$ describes the asymptotic
behavior of the global solution provided by Theorem \ref{global1}. This is the
subject of the next theorem.

\begin{theo}
\label{teoscat}\textit{(Scattering)}. Assume the conditions on Theorem \ref{global1}, and let $[u,v]$ be the solution of
(\ref{GoodBous}) provided by Theorem \ref{global1} with data $[u_{0},v_{0}]\in M^{s}_{p',q}\times M^{s}_{p',q}.$ Then there exists $[u_{0}^{\pm}%
,v_{0}^{\pm}]\in \mathcal{L}^\infty_{\alpha,s}$ such that
\begin{align}
\Vert\lbrack u(t)-u^{\pm}(t),v(t)-v^{\pm}(t)]\Vert_{M^s_{p,q}\times D^{-1}JM^s_{p,q}}  &  =O(\left\vert t\right\vert
^{1-\alpha\lambda})\ \mbox{as}\ t\rightarrow\pm\infty
,\label{aux-scat1}
\end{align}
where $[u^{+}(t),v^{+}(t)],[u^{-}(t),v^{-}(t)]$ stand for the unique global
mild solutions of the linear problem associated to (\ref{EquiSyst}) with initial
data $[u_{0}^{+},v_{0}^{+}]$ and $[u_{0}^{-},v_{0}^{-}]$, respectively.
\end{theo}
\begin{theo}
\label{teoasymp}\textit{(Stability)}. Assume the conditions on Theorem \ref{global1}, and let $[u,v], [\tilde{u},\tilde{v}]$ be the solutions of
(\ref{GoodBous}) provided by Theorem \ref{global1} with data $[u_{0},v_{0}],$ $[\tilde{u}_{0},\tilde{v}_{0}]\in M^{s}_{p',q}\times M^{s}_{p',q},$ respectively. Then 
\begin{align*}
\lim_{\vert t\vert\rightarrow \infty}(1+\vert t\vert)^\alpha\Vert B(t)[u_0-\tilde{u}_0,v_0-\tilde{v}_0]\Vert_{M^s_{p,q}\times D^{-1}JM^s_{p,q}}=0
\end{align*}
if and only if
\begin{align*}
\lim_{\vert t\vert\rightarrow\infty}(1+\vert t\vert)^\alpha\Vert \lbrack u(t)-\tilde{u}(t),v(t)-\tilde{v}(t)]\Vert_{M^s_{p,q}\times D^{-1}JM^s_{p,q}} =0.
\end{align*}
\end{theo}

\section{Linear and nonlinear estimates}\label{section2}
In this section we establish some linear estimates for the group $B$ defined in (\ref{g1}). We start denoting
\[
l_2(t)g:=J^{-1}DB_3(t)g=[\sin(t|\xi|\langle\xi\rangle)\widehat{g}(\xi)]^{\vee}.
\]
From Theorem 2.1 in Linares and Scialom  \cite{LinaresScialom},  and the proof of Lemma 2.1 and Remark 1.4 in Ferreira \cite{Lucas1}, the following estimates hold.
\begin{lemm}\label{LinSciFel}
Let $2\leq p \leq \infty$ and $i=1,2.$ Then for all $g\in L^{p'} $ it holds
\[ \|B_i(t)g\|_{L^p}\leq c |t|^{-\alpha}\|g\|_{L^{p'}},\]
\[ \|l_2(t)g\|_{L^p}\leq c |t|^{-\alpha}\|g\|_{L^{p'}},\]
with $\frac 1p+\frac{1}{p'}=1$ and $\alpha= n\left(\frac 12-\frac 1p\right).$
\end{lemm}
Using Lemma \ref{LinSciFel} we get the following lemma.
\begin{lemm}\label{LinB1B2}
Let $2\leq p \leq \infty,$  $i=1,2$  and $s\in \mathbb{R}.$ Then for all $g\in M_{p',q}^s $ it holds
\begin{align}
\|B_i(t)g\|_{M_{p,q}^s}\leq c |t|^{-\alpha}\|g\|_{M_{p',q}^s}\label{LinB1B2es}
\end{align}
with $\frac 1p+\frac{1}{p'}=1$ and $\alpha= n\left(\frac 12-\frac 1p\right).$
\end{lemm}
\proof By definition of $\square_k,$ the Hausdorff-Young inequality and Lemma \ref{LinSciFel}, we have
\begin{align*}
\|\square_k B_i(t)g\|_p&\leq c \|\sigma_k(\xi) \mathcal{F}B_i(t)(\xi) \mathcal{F}g(\xi)\|_{p'}=c \|\mathcal{F}B_i(t)\square_kg\|_{p'} \\
& \leq c \|B_i(t)\square_kg\|_{p}\leq c |t|^{-\alpha}  \|\square_kg\|_{p'}.
\end{align*}
Multiplying the last inequality by $(1+\vert k\vert)^s$ and taking the $l^q$-norm we arrive at the desired result. 
\endproof

\begin{lemm}\label{LinB1B3}
Let $2\leq p \leq \infty,$  $i=1,3$ and $s\in \mathbb{R}.$ Then, for all $g\in M_{p',q}^s $ it holds
\begin{align} 
\|B_i(t)g\|_{D^{-1}JM_{p,q}^s}\leq c |t|^{-\alpha}\|g\|_{M_{p',q}^s}\label{LinB1B3es}
\end{align}
with $\frac 1p+\frac{1}{p'}=1$ and $\alpha= n\left(\frac 12-\frac 1p\right).$
\end{lemm}
\proof By definition of $D^{-1}JM_{p',q}^s$-norm, we have
\[
 \|B_3(t)g\|_{D^{-1}JM_{p,q}^s}=  \| J^{-1}DB_3(t)g\|_{M_{p,q}^s}=\| l_2(t)g\|_{M_{p,q}^s}. 
\]
Following the proof of Lemma \ref{LinB1B2} and using Lemma \ref{LinSciFel}, we obtain
\[ \| l_2(t)g\|_{M_{p,q}^s} \leq c |t|^{-\alpha}\|g\|_{M_{p',q}^s},\]
which concludes the estimate for $B_3.$ Now, for $i=1,$ from the $L^p$-continuity of $J^{-1}D,$ we arrive at
\[
 \|B_1(t)g\|_{D^{-1}JM_{p,q}^s}=  \| J^{-1}DB_1(t)g\|_{M_{p,q}^s}\leq \| B_1(t)g\|_{M_{p,q}^s} \leq c |t|^{-\alpha}\|g\|_{M_{p',q}^s}.
\]
Thus we conclude the proof of the lemma. 
\endproof

\begin{lemm}\label{LinB1B2(1+t)}
Let $2\leq p < \infty,$  $0<q<\infty,$ $i=1,2$ and $s\in \mathbb{R}.$ Then, for all $g\in M_{p',q}^s $ it holds
\begin{align}
\|B_i(t)g\|_{M_{p,q}^s}\leq c (1+|t|)^{-\alpha}\|g\|_{M_{p',q}^s}\label{LinB1B2(1+t)es}
\end{align}
with $\frac 1p+\frac{1}{p'}=1$ and $\alpha= n\left(\frac 12-\frac 1p\right).$
\end{lemm}
\proof From Lemma \ref{LinSciFel} and since $\square_k$ and $B_i(t)$ commmute, we have
\begin{equation}\label{B12k+l1}
 \|\square_k B_i(t)g\|_{L^p}\leq c |t|^{-\alpha}\|\square_k g\|_{L^{p'}}\leq c  |t|^{-\alpha} \sum_{l\in \Lambda} \|\square_{k+l} g\|_{L^{p'}},
 \end{equation}
where $\Lambda=\{ l\in \mathbb{Z}^n: B(l,\sqrt{2n})\cap B(0,\sqrt{2n})\neq \emptyset\}.$ We also used that
\[ \|\square_k f\|_{L^{p_2}}\leq c  \sum_{l\in \Lambda} \|\square_{k+l} f\|_{L^{p_1}},\]
for all $0<p_1\leq p_2\leq \infty$ (see for instance Baoxiang {\it et al.} \cite{BaoLifBol}).\\

From H\"older's inequality we have that
 \begin{align*}
 \|\square_k B_i(t)g\|_{L^p}&=\|\mathcal{F}^{-1}\sum_{l\in \Lambda} \sigma_k(\xi)\sigma_{k+l}(\xi)\widehat{B_i}(\xi)\widehat{g}(\xi)\|_{L^{p}}\leq c  \sum_{l\in \Lambda} \|\sigma_k(\xi)\sigma_{k+l}(\xi)\widehat{B_i}(\xi)\widehat{g}(\xi)\|_{L^{p'}}\\
 &\leq c \|\sigma_k(\xi)\widehat{g}(\xi)\|_{L^{r'}} \leq c  \|\square_k g\|_{L^r} \leq c  \sum_{l\in \Lambda} \|\square_{k+l} g\|_{L^{p'}},
 \end{align*}
where $1\leq r\leq 2.$ Therefore,
\begin{equation}\label{B12k+l2}
\|\square_k B_i(t)g\|_{L^p} \leq c \sum_{l\in \Lambda} \|\square_{k+l} g\|_{L^{p'}}.
 \end{equation}
From (\ref{B12k+l1}), (\ref{B12k+l2}) and considering the cases $|t|\leq 1$ and $|t|>1,$ we arrive at
\begin{equation}\label{B12(1+t)Cua1}
\|\square_k B_i(t)g\|_{L^p} \leq c (1+|t|)^{-\alpha}\sum_{l\in \Lambda} \|\square_{k+l} g\|_{L^{p'}}.
 \end{equation}
 Multiplying $(1+\vert k\vert)^s$ and then taking the $l^q$-norm in both sides of (\ref{B12(1+t)Cua1}), we immediately obtain the desired result. 
 \endproof
 
 In a similar way as in Lemma \ref{LinB1B2(1+t)} we also obtain. 
 \begin{lemm}\label{LinB1B3(1+t)}
Let $2\leq p < \infty,$  $0<q<\infty,$ $i=1,3$ and $s\in \mathbb{R}.$ Then, for all $g\in M_{p',q}^s $ it holds
\begin{align} 
\|B_i(t)g\|_{D^{-1}JM_{p,q}^s}\leq c (1+|t|)^{-\alpha}\|g\|_{M_{p',q}^s}\label{LinB1B3(1+t)es}
\end{align}
with $\frac 1p+\frac{1}{p'}=1$ and $\alpha= n\left(\frac 12-\frac 1p\right).$
\end{lemm}

\begin{remark}\label{esti_comp}
The right-hand side of (\ref{LinB1B2es})-(\ref{LinB1B3es}) have a singularity at $t=0.$ In order to control the growth of this singularity we assume $\lambda_0(n)<\lambda<\lambda_1(n)$ ($\alpha< 1$) in Remark \ref{global2}. In estimates (\ref{LinB1B2(1+t)es})-(\ref{LinB1B3(1+t)es}) we remove the singularity but preserving the decay at $t=\infty.$
\end{remark} 

 \begin{lemm}\label{Ib1Opcion2}
 Let $1\leq p,p_1,p_2,\sigma,\sigma_1,\sigma_2\leq \infty.$  If $\frac{1}{p}=\frac{1}{p_1}+\frac{1}{p_2},$ $\frac{1}{\sigma}=\frac{1}{\sigma_1}+\frac{1}{\sigma_2}-1$ and  $s\geq 0,$ there exists $C>0,$ such that for any $u\in M^s_{p_1,\sigma_1}(\mathbb{R}^n)$ and $v\in M^s_{p_2,\sigma_2}(\mathbb{R}^n),$ it holds
 \[ \|u v\|_{M^{s}_{p,\sigma}}\leq C  \|u\|_{M^{s}_{p_1,\sigma_1}}\|v\|_{M^{s}_{p_2,\sigma_2}}.\]
 \end{lemm}
 \proof From Remark 2.4 in Iwabuchi \cite{Iwabuchi}, we have
  \[ \|u v\|_{M^{s}_{p,\sigma}}\leq C  \|u\|_{M^{s}_{p_1,\sigma_1}}\|v\|_{M^{0}_{p_2,\sigma_2}}+\|u\|_{M^{0}_{p_3,\sigma_3}}\|v\|_{M^{s}_{p_4,\sigma_4}},\]
where 
\[\frac{1}{p}=\frac{1}{p_1}+\frac{1}{p_2}=\frac{1}{p_3}+\frac{1}{p_4} \ \ \ \ \text{and}\ \ \ \ \frac{1}{\sigma}=\frac{1}{\sigma_1}+\frac{1}{\sigma_2}-1=\frac{1}{\sigma_3}+\frac{1}{\sigma_4}-1.\]
Taking $p_1=p_3,$ $p_2=p_4,$ $\sigma_1=\sigma_3$ and $\sigma_2=\sigma_4,$ we arrive at
 \begin{equation}\label{Mult1}
 \|u v\|_{M^{s}_{p,\sigma}}\leq C  \|u\|_{M^{s}_{p_1,\sigma_1}}\|v\|_{M^{0}_{p_2,\sigma_2}}+\|u\|_{M^{0}_{p_1,\sigma_1}}\|v\|_{M^{s}_{p_2,\sigma_2}},
 \end{equation}
where 
\[\frac{1}{p}=\frac{1}{p_1}+\frac{1}{p_2}\ \ \ \ \text{and}\ \ \ \ \frac{1}{\sigma}=\frac{1}{\sigma_1}+\frac{1}{\sigma_2}-1.\]
Using Proposition 2.1$(ii)$ in \cite{Iwabuchi}, we obtain
\begin{equation}\label{Mult2}
\|v\|_{M^{0}_{p_2,\sigma_2}} \apprle \|v\|_{M^{s}_{p_2,\sigma_2}}\ \  \ \text{and} \ \ \ \|u\|_{M^{0}_{p_1,\sigma_1}}\apprle \|u\|_{M^{s}_{p_1,\sigma_1}}.
\end{equation}
Combining (\ref{Mult1}) with (\ref{Mult2}) we obtain the desire result. 
 \endproof
 
 With the aim of making the reading easier, we present three lemmas which allow us to deal with the nonlinearity $f(u)=u^{\lambda}$. The proof of the first two can be found in Iwabuchi \cite{Iwabuchi} (Proposition 2.7 (ii) and Corollary 2.9 (ii)) and the proof of the third one is in  Baoxiang and Hudzik \cite{BaoxHudz} (Lemma 8.2).
 \begin{lemm} \label{Ib1}\cite{Iwabuchi}
 Let $1\leq p,p_1,p_2\leq \infty,$ $1<\sigma,\sigma_1,\sigma_2<\infty.$   If $\frac{1}{p}=\frac{1}{p_1}+\frac{1}{p_2},$ $\frac{1}{\sigma}-\frac{1}{\sigma_1}-\frac{1}{\sigma_2}+1\leq \frac{s}{n}<\frac{1}{\sigma},$ there exists $C>0,$ such that for any $u\in M^s_{p_1,\sigma_1}(\mathbb{R}^n)$ and $v\in M^s_{p_2,\sigma_2}(\mathbb{R}^n),$ it holds
 \[ \|u v\|_{M^{s}_{p,\sigma}}\leq C  \|u\|_{M^{s}_{p_1,\sigma_1}}\|v\|_{M^{s}_{p_2,\sigma_2}}.\]
 \end{lemm}

\begin{lemm}\label{Ib2}\cite{Iwabuchi}
Let $1\leq q \leq \infty, $ $p \in \mathbb{N},$ $0\leq  s <n/{\nu},$ and $1\leq \mu, \nu <\infty$ satisfy
\[ \frac{1}{\nu}-\frac{(p-1)s}{n}\leq \frac{p}{\mu}-p+1, \ \ \ \ \ 1\leq \nu\leq \mu.\]
Then, there exists $C>0$ such that for any $u\in M^{s}_{{pq},\mu}(\mathbb{R}^n),$ we have
\[ \|u^p\|_{M^{s}_{{q},\mu}}\leq C  \|u\|^p_{M^{s}_{{pq},\nu}}.\]
\end{lemm}

\begin{lemm} \label{BaoHud}\cite{BaoxHudz}
 Let $1\leq p,p_i,\gamma, \gamma_i\leq \infty,$ satisfy
 \[\frac{1}{p'}=\frac{1}{p_1}+\frac{1}{p_2}+\cdots +\frac{1}{p_m}, \ \ \ \ \frac{1}{\gamma'}=\frac{1}{\gamma_1}+\frac{1}{\gamma_2}+\cdots +\frac{1}{\gamma_m}.\]
 Let $\alpha=0$ and $q=1,$ or $\alpha>0$ and $q'\alpha>nm.$ Then we have
 
 \[ \|u_1u_2\cdots u_m\|_{l_{\square}^{-\alpha,q}(L^{\gamma'}(\mathbb{R}; L^{p'}))}\apprle \prod_{i=1}^m\|u_i\|_{l_{\square}^q(L^{\gamma_i}(\mathbb{R}; L^{p_i}))}.\]
 \end{lemm}
 
From Propositions 5.1 and 5.3 in Baoxiang and Hudzik \cite{BaoxHudz} we have the next Strichartz-type estimate in the spaces $l_{\square}^q(L^{\gamma}(\mathbb{R};L^{\sigma}(\mathbb{R}^n))).$ 
\begin{prop} \label{Strichartz1} 
Let $2\leq \sigma<\infty,$ $1\leq q<\infty$ and $\gamma \geq \max\{2, \gamma_{\sigma}\},$ where
\[\frac{2}{\gamma_{\sigma}}=n\left(\frac 12 -\frac{1}{\sigma} \right).\]
Then, for $i=1,2$ we have
\[ \Vert B_i(t)g\Vert_{l_{\square}^q(L^{\gamma}(\mathbb{R};L^{\sigma}))}\apprle \Vert g\Vert_{M_{2,q}}.\]
\[ \Vert l_2(t)g\Vert_{l_{\square}^q(L^{\gamma}(\mathbb{R};L^{\sigma}))}\apprle \Vert g\Vert_{M_{2,q}}.\]
\[ \left\Vert \int_0^tB_i(t-\tau)g(\cdot,\tau)d\tau\right\Vert_{l_{\square}^q(L^{\infty}(\mathbb{R};L^2))}\apprle \Vert g\Vert_{l_{\square}^q(L^{\gamma'}(\mathbb{R};L^{\sigma'}))}.\]
\[ \left\Vert \int_0^tB_i(t-\tau)g(\cdot,\tau)d\tau\right\Vert_{l_{\square}^q(L^{\gamma}(\mathbb{R};L^{\sigma}))}\apprle \Vert g\Vert_{l_{\square}^q(L^{\gamma'}(\mathbb{R};L^{\sigma'}))}.\]
\end{prop}

\begin{prop}\label{Strichartz2} 
Let $2\leq \sigma<\infty,$ $1\leq q<\infty$ and $\gamma \geq \max\{2, \gamma_{\sigma}\},$ where
\[\frac{2}{\gamma_{\sigma}}=n\left(\frac 12 -\frac{1}{\sigma} \right).\]
Then, for $i=1,3,$ we have
\[ \Vert B_i(t)g\Vert_{l_{\square}^q(L^{\gamma}(\mathbb{R};D^{-1}JL^{\sigma}))}\apprle \Vert g\Vert_{M_{2,q}}.\]
\[ \left\Vert \int_0^tB_1(t-\tau)g(\cdot,\tau)d\tau\right\Vert_{l_{\square}^q(L^{\infty}(\mathbb{R};D^{-1}JL^2))}\apprle \Vert g\Vert_{l_{\square}^q(L^{\gamma'}(\mathbb{R};L^{{\sigma}'}))}.\]
\[ \left\Vert \int_0^tB_1(t-\tau)g(\cdot,\tau)d\tau\right\Vert_{l_{\square}^q(L^{\gamma}(\mathbb{R};D^{-1}JL^{\sigma}))}\apprle \Vert g\Vert_{l_{\square}^q(L^{\gamma'}(\mathbb{R};L^{{\sigma}'}))}.\]
\end{prop}
\proof Recall that $ \Vert g\Vert_{l_{\square}^q(L^{\gamma}(\mathbb{R};D^{-1}JL^{\sigma}))}:=\Vert J^{-1}D g\Vert_{l_{\square}^q(L^{\gamma}(\mathbb{R};L^{\sigma}))}.$ Then, from the $L^{\sigma}$-continuity of $J^{-1}D$ and Proposition \ref{Strichartz1}, we have
\begin{align*}
\Vert B_1(t)g\Vert_{l_{\square}^q(L^{\gamma}(\mathbb{R};D^{-1}JL^{\sigma}))}&=\Vert J^{-1}DB_1(t)g\Vert_{l_{\square}^q(L^{\gamma}(\mathbb{R};L^{\sigma}))}\leq \Vert B_1(t)g\Vert_{l_{\square}
^q(L^{\gamma}(\mathbb{R};L^{\sigma}))}\\
&\apprle \Vert g\Vert_{M_{2,q}}.
\end{align*}
Similarly, we obtain
\begin{align*}
\Vert B_3(t)g\Vert_{l_{\square}^q(L^{\gamma}(\mathbb{R};D^{-1}JL^{\sigma}))}&=\Vert J^{-1}DB_3(t)g\Vert_{l_{\square}^q(L^{\gamma}(\mathbb{R};L^{\sigma}))}=\Vert l_2(t)g\Vert_{l_{\square}^q(L^{\gamma}(\mathbb{R};L^{\sigma}))}\\
&\apprle \Vert g\Vert_{M_{2,q}}.
\end{align*}
Again, from Proposition \ref{Strichartz1} we have
\begin{align*}
\left\Vert \int_0^tB_1(t-\tau)g(\cdot,\tau)d\tau\right\Vert_{l_{\square}^q(L^{\infty}(\mathbb{R};D^{-1}JL^2))}&=\left\Vert \int_0^tJ^{-1}DB_1(t-\tau)g(\cdot,\tau)d\tau\right\Vert_{l_{\square}^q(L^{\infty}(\mathbb{R};L^2))}\\
&\leq \left\Vert \int_0^tB_1(t-\tau)g(\cdot,\tau)d\tau\right\Vert_{l_{\square}^q(L^{\infty}(\mathbb{R};L^2))}\\
&\apprle \Vert g\Vert_{l_{\square}^q(L^{\gamma'}(\mathbb{R};L^{{\sigma}'}))}.
\end{align*}
By the $L^{\sigma}$-continuity of $J^{-1}D$ and Proposition \ref{Strichartz1}, we obtain
\begin{align*}
\left\Vert \int_0^tB_1(t-\tau)g(\cdot,\tau)d\tau\right\Vert_{l_{\square}^q(L^{\gamma}(\mathbb{R};D^{-1}JL^{\sigma}))}&=\left\Vert \int_0^tJ^{-1}DB_1(t-\tau)g(\cdot,\tau)d\tau\right\Vert_{l_{\square}^q(L^{\gamma}(\mathbb{R};L^{\sigma}))}\\
&\leq \left\Vert \int_0^tB_2(t-\tau)g(\cdot,\tau)d\tau\right\Vert_{l_{\square}^q(L^{\gamma}(\mathbb{R};L^{\sigma}))}\apprle  \Vert g\Vert_{l_{\square}^q(L^{\gamma'}(\mathbb{R};L^{{\sigma}'}))}.
\end{align*}
This finish the proof of the proposition. 
\endproof
\section{Proofs}
In this section we prove the results established in Section 1. We start estimating the nonlinear part of (\ref{ms5}). For this, let us define the operators
\[\Phi_1[u(t),v(t)]=B_1(t)u_0+B_2(t)v_0-\int_0^tB_2(t-\tau)f(u(\tau))d\tau,\]
\[\Phi_2[u(t),v(t)]=B_3(t)u_0+B_1(t)v_0-\int_0^tB_1(t-\tau)f(u(\tau))d\tau,\]
and $\Phi[u,v]=[\Phi_1(u,v), \Phi_2(u,v) ].$
\begin{prop}\label{non_est} Asume that $\lambda>1$ is a positive integer,  $1\leq q<2,$ $p=\lambda+1, $ and $n-\frac nq \leq s < \frac nq,$ and $\lambda>\lambda_0(n).$ If $q=1$ assume also that $\lambda\geq 2.$ There exists a constant $C_1>0$ such that
\begin{align}
&\Vert \Phi[u,v]-\Phi[\tilde{u},\tilde{v}]\Vert_{\mathcal{L}^\infty_{\alpha,s}}\leq C_1\sup_{-\infty<t<\infty}(1+ \vert t\vert)^{\alpha}\Vert (u-\tilde{u})(t)\Vert_{M^s_{p,q}}\nonumber\\
&\ \ \ \times\sup_{-\infty<t<\infty}(1+ \vert t\vert)^{\alpha(\lambda-1)}\sum_{k=1}^{\lambda}\Vert u(t)\Vert^{\lambda-k}_{M^s_{p,q}}\Vert \tilde{u}(t)\Vert^{k-1}_{M^s_{p,q}}.\label{ms4}
\end{align}
\end{prop}
\begin{proof}
Without loss of generality we assume $t>0.$ First suppose that  $1<q<2.$ Since $n-\frac nq\leq s<\frac nq,$ from Lemmas \ref{LinB1B2(1+t)}, \ref{LinB1B3(1+t)}, \ref{Ib1}, \ref{Ib2}, we obtain
\begin{align}
&\Vert \Phi[u,v]-\Phi[\tilde{u},\tilde{v}]\Vert_{M^s_{p,q}}+\Vert \Phi[u,v]-\Phi[\tilde{u},\tilde{v}]\Vert_{D^{-1}JM^s_{p,q}}\nonumber\\
&\ \ \ \leq\int_0^t\Vert B_2(t-\tau)(f(u)-f(\tilde{u}))\Vert_{M^s_{p,q}}d\tau+\int_0^t\Vert B_1(t-\tau)(f(u)-f(\tilde{u}))\Vert_{D^{-1}JM^s_{p,q}}d\tau\nonumber\\
&\ \ \ \apprle\int_0^t(1+ \vert t-\tau\vert)^{-\alpha}\left\Vert (u-\tilde{u})\left( \sum_{k=1}^{\lambda}u^{\lambda-k} \tilde{u}^{k-1}\right)\right\Vert_{M^s_{p',q}}d\tau\nonumber\\ 
&\ \ \ \apprle\int_0^t(1+ \vert t-\tau\vert)^{-\alpha}\Vert u-\tilde{u}\Vert_{M^s_{p,q}}\sum_{k=1}^{\lambda}\Vert u^{\lambda-k}\Vert_{M^s_{\frac{p}{\lambda-k},q}}\Vert \tilde{u}^{k-1}\Vert_{M^s_{\frac{p}{k-1},q}}d\tau\nonumber\\ 
&\ \ \ \apprle\int_0^t(1+ \vert t-\tau\vert)^{-\alpha}\Vert u-\tilde{u}\Vert_{M^s_{p,q}}\sum_{k=1}^{\lambda}\Vert u\Vert^{\lambda-k}_{M^s_{p,q}}\Vert \tilde{u}\Vert^{k-1}_{M^s_{p,q}}d\tau\nonumber\\ 
&\ \ \ \apprle\sup_{0<t<\infty}(1+ t)^{\alpha}\Vert u-\tilde{u}\Vert_{M^s_{p,q}}\sum_{k=1}^{\lambda}\sup_{0<t<\infty}(1+  \tau)^{\alpha(\lambda-k)}\Vert u\Vert^{\lambda-k}_{M^s_{p,q}}\sup_{0<t<\infty}(1+ \tau)^{\alpha(k-1)}\Vert \tilde{u}\Vert^{k-1}_{M^s_{p,q}}\nonumber\\
&\ \ \ \times \int_0^t(1+t-\tau)^{-\alpha}(1+\tau)^{-\alpha\lambda}d\tau.\label{mod1qno1}
\end{align}
Since $\lambda>\lambda_0(n),$ then $\alpha\lambda> 1;$ thus, it follows that
\begin{align}
\int_0^{t/2}(1+t-\tau)^{-\alpha}(1+ \tau)^{-\alpha\lambda}d\tau \apprle (1+t)^{-\alpha}\int_0^{t/2}(1+ \tau)^{-\alpha\lambda}d\tau\apprle (1+ t)^{-\alpha}.\label{mod2qno1}
\end{align}
Also, it is straightforward to get
\begin{align}
\int_{t/2}^t(1+ t-\tau)^{-\alpha}(1+ \tau)^{-\alpha\lambda}d\tau&= \int_{0}^{t/2}(1+ \tau)^{-\alpha}(1+ t-\tau)^{-\alpha\lambda}d\tau \nonumber\\
&\ \apprle(1+t)^{-\alpha\lambda}\int_0^{t/2}(1+ \tau)^{-\alpha}d\tau  \apprle (1+ t)^{-\alpha}.\label{mod3qno1}
\end{align}
Indeed, for $0<\lambda<1$, since $\alpha\lambda> 1,$ it holds 
$$(1+t)^{-\alpha\lambda}\int_0^{t/2}(1+ \tau)^{-\alpha}d\tau = \frac{(1+t)^{-\alpha\lambda}}{1-\alpha}\left[ (1+\frac{t}{2})^{1-\alpha}-1\right] \apprle (1+ t)^{-\alpha}.$$
For $\alpha>1,$ it holds
$$(1+t)^{-\alpha\lambda}\int_0^{t/2}(1+ \tau)^{-\alpha}d\tau = \frac{(1+t)^{-\alpha\lambda}}{\alpha-1}\left[(1+\frac{t}{2})^{1-\alpha}-1\right] \apprle (1+ t)^{-\alpha}.$$
If $\alpha=1$ and $\lambda\geq 2,$ it holds
$$(1+t)^{-\alpha\lambda}\int_0^{t/2}(1+ \tau)^{-\alpha}d\tau \apprle (1+t)^{-1}\frac{\ln(1+t)}{(1+t)^{\lambda-1}}\apprle (1+t)^{-1}(1+t)^{2-\lambda}\apprle (1+t)^{-1}.$$
Finally, if $\alpha=1$ and $1<\lambda<2,$ it holds
$$(1+t)^{-\alpha\lambda}\int_0^{t/2}(1+ \tau)^{-\alpha}d\tau \apprle (1+t)^{-1}\frac{\ln(1+t)}{(1+t)^{\lambda-1}}\apprle (1+t)^{-1}\frac{1}{\lambda-1}\apprle (1+t)^{-1}.$$

From (\ref{mod1qno1})-(\ref{mod3qno1}) we conclude the proof of (\ref{ms4}) for $1<q<2.$ The proof in the case $q=1$ is the same, just using Lemma \ref{Ib1Opcion2} instead of Lemma \ref{Ib1}.
\end{proof}
\begin{remark}\label{rem2}
Assume $p, q$ and $s$ as in Proposition \ref{non_est}. If $\lambda_0(n)<\lambda <\lambda_1(n),$ there exists a constant $C_2>0$ such that
\begin{align}
&\Vert \Phi[u,v]-\Phi[\tilde{u},\tilde{v}]\Vert_{\mathcal{H}^\infty_{\beta,s}}\leq C_2\sup_{-\infty<t<\infty}\vert t\vert^{\beta}\Vert u-\tilde{u}\Vert_{M^s_{p,q}}\nonumber\\
&\ \ \ \times\sup_{-\infty<t<\infty}\vert t\vert^{\beta(\lambda-1)}\sum_{k=1}^{\lambda}\Vert u(t)\Vert^{\lambda-k}_{M^s_{p,q}}\Vert \tilde{u}(t)\Vert^{k-1}_{M^s_{p,q}}.\label{ms5A}
\end{align}
Indeed, using Lemmas \ref{LinB1B2}, \ref{LinB1B3}, \ref{Ib1Opcion2}, \ref{Ib1}, \ref{Ib2}, and following  the proof of Proposition \ref{non_est} we get
\begin{align}
&\Vert \Phi[u,v]-\Phi[\tilde{u},\tilde{v}]\Vert_{M^s_{p,q}}+\Vert \Phi[u,v]-\Phi[\tilde{u},\tilde{v}]\Vert_{D^{-1}JM^s_{p,q}}\nonumber\\
&\ \ \ \apprle\int_0^t(t-\tau)^{-\alpha}\Vert u-\tilde{u}\Vert_{M^s_{p,q}}\sum_{k=1}^{\lambda}\Vert u\Vert^{\lambda-k}_{M^s_{p,q}}\Vert \tilde{u}\Vert^{k-1}_{M^s_{p,q}}d\tau\nonumber\\ 
&\ \ \ \apprle\sup_{0<t<\infty}t^{\beta}\Vert  u-\tilde{u}\Vert_{M^s_{p,q}}\sum_{k=1}^{\lambda}\sup_{0<t<\infty}t^{\beta(\lambda-k)}\Vert u(t)\Vert^{\lambda-k}_{M^s_{p,q}}\sup_{0<t<\infty}t^{\beta(k-1)}\Vert \tilde{u}(t)\Vert^{k-1}_{M^s_{p,q}} \nonumber\\
&\ \ \ \times \int_0^t(t-\tau)^{-\alpha}\tau^{-\beta\lambda}d\tau.\label{mod1}
\end{align}
Since $1-\beta(\lambda-1)-\alpha=0,$  multiplying by $t^{\beta}$ and taking the supremum over $\mathbb{R}$ we conclude (\ref{ms5A}). Notice the the condition $\lambda_0(n)<\lambda <\lambda_1(n)$ is due the integrability of the beta function in (\ref{mod1}).
\end{remark}
\begin{prop}\label{non_estLoc} Asume that $\lambda>1$ is a positive integer,  $1\leq q<2,$ $p=\lambda+1 and  $  $n-\frac nq \leq s < \frac nq.$  If $q=1$ assume also that $\lambda\geq 2.$
Then there exists a constant $C_3>0$ such that
\begin{align}
&\Vert \Phi[u,v]-\Phi[\tilde{u},\tilde{v}]\Vert_{\mathcal{L}^T_{\alpha,s}}\leq C_3 T \sup_{-T<t<T}(1+ \vert t\vert)^{\alpha}\Vert (u-\tilde{u})(t)\Vert_{M^s_{p,q}}\nonumber\\
&\ \ \ \times\sup_{-T<t<T}(1+ \vert t\vert)^{\alpha(\lambda-1)}\sum_{k=1}^{\lambda}\Vert u(t)\Vert^{\lambda-k}_{M^s_{p,q}}\Vert \tilde{u}(t)\Vert^{k-1}_{M^s_{p,q}}.\label{ms4Loc}
\end{align}
\end{prop}
\begin{proof}
Without loss of generality we assume $t>0.$ First suppose that  $1<q<2.$ Since $n-\frac nq\leq s<\frac nq,$ from Lemmas \ref{LinB1B2(1+t)}, \ref{LinB1B3(1+t)}, \ref{Ib1}, \ref{Ib2}, we obtain
\begin{align}
&\Vert \Phi[u,v]-\Phi[\tilde{u},\tilde{v}]\Vert_{M^s_{p,q}}+\Vert \Phi[u,v]-\Phi[\tilde{u},\tilde{v}]\Vert_{D^{-1}JM^s_{p,q}}\nonumber\\
&\ \ \ \leq\int_0^t\Vert B_2(t-\tau)(f(u)-f(\tilde{u}))\Vert_{M^s_{p,q}}d\tau+\int_0^t\Vert B_1(t-\tau)(f(u)-f(\tilde{u}))\Vert_{D^{-1}JM^s_{p,q}}d\tau\nonumber\\
&\ \ \ \apprle\int_0^t(1+ \vert t-\tau\vert)^{-\alpha}\left\Vert (u-\tilde{u})\left( \sum_{k=1}^{\lambda}u^{\lambda-k} \tilde{u}^{k-1}\right)\right\Vert_{M^s_{p',q}}d\tau\nonumber\\ 
&\ \ \ \apprle\int_0^t(1+ \vert t-\tau\vert)^{-\alpha}\Vert u-\tilde{u}\Vert_{M^s_{p,q}}\sum_{k=1}^{\lambda}\Vert u^{\lambda-k}\Vert_{M^s_{\frac{p}{\lambda-k},q}}\Vert \tilde{u}^{k-1}\Vert_{M^s_{\frac{p}{k-1},q}}d\tau\nonumber\\ 
&\ \ \ \apprle\int_0^t(1+ \vert t-\tau\vert)^{-\alpha}\Vert u-\tilde{u}\Vert_{M^s_{p,q}}\sum_{k=1}^{\lambda}\Vert u\Vert^{\lambda-k}_{M^s_{p,q}}\Vert \tilde{u}\Vert^{k-1}_{M^s_{p,q}}d\tau\nonumber\\ 
&\ \ \ \apprle\sup_{0<t<T}(1+  t)^{\alpha}\Vert u-\tilde{u}\Vert_{M^s_{p,q}}\sum_{k=1}^{\lambda}\sup_{0<t<T}(1+  \tau)^{\alpha(\lambda-k)}\Vert u\Vert^{\lambda-k}_{M^s_{p,q}}\sup_{0<t<T}(1+  \tau)^{\alpha(k-1)}\Vert \tilde{u}\Vert^{k-1}_{M^s_{p,q}}\nonumber\\
&\ \ \ \times \int_0^t(1+ t-\tau)^{-\alpha}(1+ \tau)^{-\alpha\lambda}d\tau.\label{mod1qno1Loc}
\end{align}
Since $\alpha\lambda>0,$   we obtain 
\begin{align}
\int_0^{t/2}(1+ t-\tau)^{-\alpha}(1+\tau)^{-\alpha\lambda}d\tau &\apprle (1+t)^{-\alpha}\int_0^{t/2}(1+ \tau)^{-\alpha\lambda}d\tau\apprle (1+ t)^{-\alpha} T.\label{mod2qno1}
\end{align} 
Using that  $\alpha>0$ and $\alpha(\lambda-1)>0$  we arrive at 
\begin{align}
&\int_{t/2}^t(1+ t-\tau)^{-\alpha}(1+ \tau)^{-\alpha\lambda}d\tau=\int_{0}^{t/2}(1+ \tau)^{-\alpha}(1+ t-\tau)^{-\alpha\lambda}d\tau\apprle (1+t)^{-\alpha\lambda}\int_0^{t/2}(1+ \tau)^{-\alpha}d\tau\nonumber\\
&\apprle (1+t)^{-\alpha}(1+t)^{-\alpha(\lambda-1)}\int_0^{t/2}(1+ \tau)^{-\alpha}d\tau \apprle (1+ t)^{-\alpha} T.\label{mod3qno1Loc}
\end{align}
From (\ref{mod1qno1Loc})-(\ref{mod3qno1Loc}) we conclude the proof of (\ref{ms4Loc}) for $1<q<2.$ The proof in the case $q=1$ is the same, just using Lemma \ref{Ib1Opcion2} instead of Lemma \ref{Ib1}.
\end{proof}
\subsection {Proof of Theorem \ref{global1}}
We prove that the mapping $\Phi[u,v]=[\Phi_1(u,v), \Phi_2(u,v) ]$ defines a contraction in the metric space $\mathcal{B}_{2\epsilon}=\{[u,v]\in \mathcal{L}^\infty_{\alpha,s}:\ \Vert [u,v]\Vert_{\mathcal{L}^\infty_{\alpha,s}}\leq 2\epsilon\},$ for some $\epsilon>0.$ From Proposition \ref{non_est} and applying the Young inequality in each term $\Vert u(t)\Vert^{\lambda-k}_{M^s_{p,q}}\Vert \tilde{u}(t)\Vert^{k-1}_{M^s_{p,q}},$ $k=1,...,\lambda,$ we have
\begin{align}
\Vert \Phi[u,v]-\Phi[\tilde{u},\tilde{v}]\Vert_{\mathcal{L}^\infty_{\alpha,s}}&\leq C_1\Vert [u,v]-[\tilde{u},\tilde{v}]\Vert_{\mathcal{L}^\infty_{\alpha,s}}(\Vert [u,v]\Vert^{\lambda-1}_{\mathcal{L}^\infty_{\alpha,s}}+\Vert [\tilde{u},\tilde{v}]\Vert^{\lambda-1}_{\mathcal{L}^\infty_{\alpha,s}})\nonumber\\
&\leq 2^\lambda\epsilon^{\lambda-1}C_1\Vert [u,v]-[\tilde{u},\tilde{v}]\Vert_{\mathcal{L}^\infty_{\alpha,s}}.\label{ms7}
\end{align}
Using (\ref{ms7}) with $[\tilde{u},\tilde{v}]=[0,0],$ Lemmas \ref{LinB1B2(1+t)} and \ref{LinB1B3(1+t)}, we have
\begin{align}
&\Vert \Phi[u,v]\Vert_{\mathcal{L}^\infty_{\alpha,s}}\leq \sup_{0<t<\infty}(1+t)^{\alpha}\Vert B_1(t)u_0+B_2(t)v_0\Vert_{M^s_{p,q}}\nonumber\\
&\ \ \ +\sup_{0<t<\infty}(1+t)^{\alpha}\Vert B_3(t)u_0+B_1(t)v_0\Vert_{D^{-1}JM^s_{p,q}}+C_12^\lambda\epsilon^\lambda\nonumber\\
&\ \ \ \leq \tilde{C}(\Vert u_0\Vert{M^s_{p',q}}+ \Vert v_0\Vert{M^s_{p',q}})+2^\lambda\epsilon^\lambda C_1\leq 2\epsilon,\label{ms8}
\end{align}
provided $2^\lambda\epsilon^{\lambda-1}C_1<1$ and $[u,v]\in \mathcal{B}_{2\epsilon}.$ From (\ref{ms7}) and (\ref{ms8}) we get that $\Phi$ is a contraction, which implies that the integral equation (\ref{ms5}) has a unique solution $[u,v]\in\mathcal{L}^\infty_{\alpha,s}$, satisfying $\Vert [u,v]\Vert_{\mathcal{L}^\infty_{\alpha,s}}\leq{2\epsilon}.$ Finally, we prove that the data-solution map is locally Lipschitz. Let $[u,v],$ $[\tilde{u},\tilde{v}]$ global solutions of (\ref{ms5}) with initial data 
$[u_0,v_0]$ and $[\tilde{u}_0,\tilde{v}_0]\in M^s_{p',q},$ respectively, satisfying $\Vert[u,v]\Vert_{\mathcal{L}^\infty_{\alpha,s}}, \Vert[\tilde{u},\tilde{v}]\Vert_{\mathcal{L}^\infty_{\alpha,s}}\leq 2\epsilon.$ Then, taking the difference between the integral equations and the ${\mathcal{L}^\infty_{\alpha,s}}$-norm, we get
\begin{align}
\Vert [u-\tilde{u},v-\tilde{v}]\Vert_{\mathcal{L}^\infty_{\alpha,s}} &\leq \Vert B(t)[u_0-\tilde{u}_0,v_0-\tilde{v}_0]\Vert_{\mathcal{L}^\infty_{\alpha,s}}\nonumber\\
&+\Vert [B_2(t-\tau)(f(u)-f(\tilde{u})),B_1(t-\tau)(f(u)-f(\tilde{u}))]\Vert_{\mathcal{L}^\infty_{\alpha,s}}\nonumber\\
&\leq \tilde{C}\Vert [u_0-\tilde{u}_0,v_0-\tilde{v}_0]\Vert_{{M^s_{p',q}}\times M^s_{p',q}} +2^\lambda \epsilon^{\lambda-1}C_1\Vert [u-\tilde{u},v-\tilde{v}]\Vert_{\mathcal{L}^\infty_{\alpha,s}}.\label{f23}
\end{align}
Therefore, since $2^\lambda \epsilon^{\lambda-1}C_1<1,$ from (\ref{f23}) we conclude the result.
\endproof
\begin{remark}\label{rem4}
Assume that $[u_0,v_0]$ satisfies (\ref{est14})-(\ref{est15}). Then, from (\ref{ms5A}) with $[\tilde{u},\tilde{v}]=[0,0],$ we have
\begin{align}
& \sup_{-\infty <t<\infty}\vert t\vert ^{\beta}\Vert \Phi[u,v]\Vert_{M^s_{p,q}}\leq \sup_{-\infty <t<\infty}\vert t\vert ^{\beta}\Vert B_1(t)u_0+B_2(t)v_0\Vert_{M^s_{p,q}}\nonumber\\
&\ \ \ +\sup_{0<t<\infty}t^{\beta}\Vert B_3(t)u_0+B_1(t)v_0\Vert_{D^{-1}JM^s_{p,q}}+2^\lambda\epsilon^\lambda C_2\nonumber\\
&\ \ \ \leq \epsilon+2^\lambda\epsilon^\lambda C_2\leq 2\epsilon,\label{ms9}
\end{align}
provided $2^\lambda\epsilon^{\lambda-1}C_2<1$ and $[u,v]\in \mathcal{B}_{2\epsilon}.$ From (\ref{ms5A}) and (\ref{ms9}) we get that $\Phi$ is a contraction, which implies that the integral equation (\ref{ms5}) has a unique solution $[u,v]\in {\mathcal{H}^\infty_{\beta,s}}.$
\end{remark}

\subsection{Proof of Theorem \ref{global4}}
We start by considering the following function space
\[X= l_{\square}^1(L^{\infty}(\mathbb{R};L^2))\cap  l_{\square}^1(L^p(\mathbb{R};L^p)).\]
Since  $\|\square_kB_1(t)u_0\|_{L^2}=\|\square_ku_0\|_{L^2},$ $\|\square_kB_2(t)v_0\|_{L^2}=\|\square_kv_0\|_{L^2},$ using Proposition \ref{Strichartz1} and taking $p$ a positive integer such that $p\leq \lambda+1,$ we have
\begin{align*}
\Vert \Phi_1(u,v)\Vert_{l_{\square}^1(L^{\infty}(\mathbb{R};L^2))}&\apprle \Vert B_1(t)u_0\Vert_{l_{\square}^1(L^{\infty}(\mathbb{R};L^2))}+\Vert B_2(t)v_0\Vert_{l_{\square}^1(L^{\infty}(\mathbb{R};L^2))}\\
&+  \left\Vert \int_0^tB_2(t-\tau)f(u(\tau))d\tau\right\Vert_{l_{\square}^1(L^{\infty}(\mathbb{R};L^2))}\\
&\apprle \Vert u_0\Vert_{M_{2,1}}+\Vert v_0\Vert_{M_{2,1}}+\Vert f(u)\Vert_{l_{\square}^1(L^{\gamma'}(\mathbb{R};L^{p'}))}\\
&\apprle \Vert u_0\Vert_{M_{2,1}}+\Vert v_0\Vert_{M_{2,1}}+\Vert u\Vert^{p-1}_{l_{\square}^1(L^p(\mathbb{R};L^p))}\Vert u\Vert^{\lambda +1-p}_{l_{\square}^1(L^{\infty}(\mathbb{R};L^{\infty}))}.
\end{align*}
In the last inequality we use Lemma \ref{BaoHud} with $\alpha=0,$ $q=1,$ $p_i=\gamma_i=p,$ for $i=1,2,...,p-1$ and $p_i=\gamma_i=\infty,$ for $i=p,p+1,...,\lambda+1.$
Next, using the fact that $\Vert \square_k u\Vert_{\infty} \apprle \Vert \square_k u\Vert_2,$ for all $k\in \mathbb{Z}^n,$ we obtain
\begin{align*}
\Vert \Phi_1(u,v)\Vert_{l_{\square}^1(L^{\infty}(\mathbb{R};L^2))}&\apprle \Vert u_0\Vert_{M_{2,1}}+\Vert v_0\Vert_{M_{2,1}}+\Vert u\Vert^{p-1}_{l_{\square}^1(L^p(\mathbb{R};L^p))}\Vert u\Vert^{\lambda +1-p}_{l_{\square}^1(L^{\infty}(\mathbb{R};L^2))}.
\end{align*}
Therefore,
\begin{align}\label{des1}
\Vert \Phi_1(u,v)\Vert_{l_{\square}^1(L^{\infty}(\mathbb{R};L^2))}&\apprle \Vert u_0\Vert_{M_{2,1}}+\Vert v_0\Vert_{M_{2,1}}+\Vert u\Vert^{\lambda}_{X}.
\end{align}
Now, we apply Proposition  \ref{Strichartz1} with $\sigma=\gamma =p$ and $q=1.$ Since $p\geq 2,$ the condition $p\geq \gamma_p,$  implies that $p\geq \max\{2,\gamma_p\}.$  For $p\neq 2,$ we obtain that
\[p\geq \gamma_p \iff \frac{2}{p} \leq \frac{2}{\gamma_p}= n\left(\frac 12-\frac 1p\right)=\frac{n(p-2)}{2p}\iff 2+\frac{4}{n}\leq p. \]
Then, from Proposition \ref{Strichartz1} and Lemma \ref{BaoHud} we obtain
\begin{align}\label{des2}
\Vert \Phi_1(u,v)\Vert_{l_{\square}^1(L^{p}(\mathbb{R};L^p))}&\apprle \Vert u_0\Vert_{M_{2,1}}+\Vert v_0\Vert_{M_{2,1}}+\Vert u\Vert^{\lambda}_{X}.
\end{align}
In order to deal with the variable $v,$ we consider the space
\[Y= l_{\square}^1(L^{\infty}(\mathbb{R};D^{-1}JL^2)).\]
Using that $\|\square_kB_1(t)u_0\|_{D^{-1}JL^2}=\|\square_ku_0\|_{L^2},$ $\|\square_kB_3(t)v_0\|_{D^{-1}JL^2}=\|\square_kv_0\|_{L^2},$ Proposition \ref{Strichartz2} and Lemma \ref{BaoHud}, we obtain
\begin{align}\label{des3}
\Vert \Phi_2(u,v)\Vert_{l_{\square}^1(L^{\infty}(\mathbb{R};D^{-1}JL^2))}&\apprle \Vert B_3(t)g\Vert_{l_{\square}^1(L^{\infty}(\mathbb{R};D^{-1}JL^2))}+\Vert B_1(t)g\Vert_{l_{\square}^1(L^{\infty}(\mathbb{R};D^{-1}JL^2))}\nonumber\\
&+  \left\Vert \int_0^tB_1(t-\tau)f(u(\tau))d\tau\right\Vert_{l_{\square}^1(L^{\infty}(\mathbb{R};D^{-1}JL^2))}\nonumber\\
&\apprle \Vert u_0\Vert_{M_{2,1}}+\Vert v_0\Vert_{M_{2,1}}+\Vert f(u)\Vert_{l_{\square}^1(L^{\gamma'}(\mathbb{R};L^{p'}))}\nonumber\\
&\apprle \Vert u_0\Vert_{M_{2,1}}+\Vert v_0\Vert_{M_{2,1}}+\Vert u\Vert^{\lambda}_{X}.
\end{align}
We prove that the mapping $\Phi[u,v]=[\Phi_1(u,v), \Phi_2(u,v) ]$ defines a contraction in the metric space $\mathcal{B}_{2\epsilon}=\{[u,v]: \Vert u\Vert_X+ \Vert v\Vert_Y\leq 2\epsilon\},$ for some $\epsilon>0.$ Following (\ref{des1}), (\ref{des2}) and (\ref{des3}) we get
\begin{align*}
&\Vert \Phi[u,v]-\Phi[\tilde{u},\tilde{v}]\Vert_{X\times Y}= \Vert \Phi_1[u,v]-\Phi_1[\tilde{u},\tilde{v}]\Vert_{X}+\Vert \Phi_2[u,v]-\Phi_2[\tilde{u},\tilde{v}]\Vert_{Y}\\
&\ \ \ \  = \left\Vert \int_0^tB_2(t-\tau)(f(u)-f(\tilde{u}))d\tau\right\Vert_{l_{\square}^1(L^{\infty}(\mathbb{R};L^2))}+\left\Vert \int_0^tB_2(t-\tau)(f(u)-f(\tilde{u}))d\tau\right\Vert_{l_{\square}^1(L^{p}(\mathbb{R};L^p))}\\
&\ \ \ \  +\left\Vert \int_0^tB_1(t-\tau)(f(u)-f(\tilde{u}))d\tau\right\Vert_{Y}\\
&\ \ \ \ \leq C \Vert u-\tilde{u}\Vert_X(\Vert u\Vert^{\lambda-1}_X+\Vert \tilde{u}\Vert^{\lambda-1}_X)\\
&\ \ \ \ \leq C2^\lambda\epsilon^{\lambda-1} \Vert [u,v]-[\tilde{u},\tilde{v}]\Vert_{X\times Y}.
\end{align*}
Also,
\begin{align*}
&\Vert \Phi[u,v]\Vert_{X\times Y}\leq \Vert u_0\Vert_{M_{2,1}}+\Vert u\Vert^{\lambda}_X\leq \Vert u_0\Vert_{M_{2,1}}+C2^\lambda\epsilon^\lambda\leq 2\epsilon,
\end{align*}
provided $2^\lambda\epsilon^{\lambda-1}C<1.$ Thus, $\Phi$ defines a contraction, which implies that the integral equation (\ref{ms5}) has a solution $[u,v]\in X\times Y.$ Notice that $l_{\square}^1(L^{\infty}(\mathbb{R};L^2))\subset L^\infty(\mathbb{R};M_{2,1})$ and $Y\subset L^\infty(\mathbb{R};D^{-1}JM_{2,1}).$ Thus, it remains
to prove the time-continuity of the solution. For that, let $t_0\in \mathbb{R},$ we show that
\begin{equation*}
\lim_{t\rightarrow t_0}\Vert [u(t),v(t)]-[u(t_0),v(t_0)]\Vert_{M_{2,1}\times D^{-1}JM_{2,1}}=0.
\end{equation*} 
From the integral equation (\ref{ms5}) we get
\begin{equation}\label{int_u}
u(t_0)=B_1(t_0)u_0+B_2(t_0)v_0-\int_0^{t_0}B_2(t_0-\tau)f(u(\tau))d\tau.
\end{equation}
Then, taking the $M_{2,1}$-norm of the difference between the first equation of (\ref{ms5}) and (\ref{int_u}) we obtain
\begin{equation*}
\Vert u(t)-u(t_0)\Vert_{M_{2,1}}\leq \Vert [B_1(t)-B_1(t_0)]u_0\Vert_{M_{2,1}}+\Vert [B_2(t)-B_2(t_0)]v_0\Vert_{M_{2,1}}+ \Vert I(t,t_0)\Vert_{M_{2,1}},
\end{equation*}
 where $I(t,t_0)$ is given by
 \[I(t,t_0)=\int_0^tB_2(t-\tau)f(u(\tau))d\tau-\int_0^{t_0}B_2(t_0-\tau)f(u(\tau))d\tau.\]
From Plancherel's theorem we obtain,
\[\Vert [B_1(t)-B_1(t_0)]u_0\Vert_{M_{2,1}}=\sum_{k\in \mathbb{Z}^n}\Vert [B_1(t)-B_1(t_0)]\square_k u_0\Vert_{L^2}=\sum_{k\in \mathbb{Z}^n}\Vert [\cos(t|\xi|\langle\xi\rangle)-\cos(t_0|\xi|\langle\xi\rangle)]\widehat{\square_k u_0}\Vert_{L^2}.\]
Since $u_0\in M_{2,1}$ and $t\mapsto \cos(t|\xi|\langle\xi\rangle)$ is a continuous function, we have that
\begin{equation*}
\Vert [B_1(t)-B_1(t_0)]u_0\Vert_{M_{2,1}}\to  0, \ \ \ \ \text{as}\ \ \ \ t\to t_0.
\end{equation*}
In a very similar way, we arrive at
\begin{equation}\label{Cont3}
\Vert [B_2(t)-B_2(t_0)]v_0\Vert_{M_{2,1}}\to  0, \ \ \ \ \text{as}\ \ \ \ t\to t_0.
\end{equation}
Now, we deal with $\Vert I(t,t_0)\Vert_{M_{2,1}}.$ For that, notice that
\begin{align*}
\Vert I(t,t_0)\Vert_{M_{2,1}}&=\left\Vert \left( B_2(t)-B_2(t_0)\right)\int_0^{t_0}B_2(-\tau)f(u(\tau))d\tau+\int_{t_0}^tB_2(t-\tau)f(u(\tau))d\tau\right\Vert_{M_{2,1}}\\
&\leq \left\Vert \left( B_2(t)-B_2(t_0)\right)\int_0^{t_0}B_2(-\tau)f(u(\tau))d\tau\right\Vert_{M_{2,1}}+\left\Vert \int_{t_0}^tB_2(t-\tau)f(u(\tau))d\tau\right\Vert_{M_{2,1}}\\
&=I_1+I_2.
\end{align*}
Taking into account that $u\in X, $ we get
\begin{align*}
\left\Vert\int_0^{t_0}B_2(-\tau)f(u(\tau))d\tau\right\Vert_{M_{2,1}}&\leq \left\Vert \int_0^{t_0}B_2(-\tau)f(u(\tau))d\tau\right\Vert_{l^1_{\square}(L^{\infty}(\mathbb{R};L^2))}\\
&\leq C \left\Vert f(u(\tau))\right\Vert_{l^1_{\square}(L^{\gamma'}(\mathbb{R};L^{p'}))}<\infty.
\end{align*}
Therefore, analogously to (\ref{Cont3}) we obtain  $\lim\limits_{t\rightarrow t_0}I_1=0.$\\

Now, using Proposition \ref{Strichartz1} we arrive at
\begin{align*}
I_2 &\leq \int_{t_0}^{t} \Vert B_2(t-\tau)f(u(\tau))\Vert_{M_{2,1}}d\tau\leq  \int_{t_0}^{t} \Vert B_2(t-\tau)f(u(\tau))\Vert_{l^1_{\square}(L^{\infty}(\mathbb{R};L^2))}d\tau \leq \int_{t_0}^{t} \Vert f(u(\tau))\Vert_{M_{2,1}}d\tau\\
&\leq \int_{t_0}^{t} \Vert u(\tau)\Vert^{\lambda}_{M_{2,1}}d\tau\leq \int_{t_0}^{t} \Vert u(\tau)\Vert^{\lambda}_{l^1_{\square}(L^{\infty}(\mathbb{R};L^2))}d\tau\leq C\vert t-t_0\vert\epsilon^\lambda\rightarrow 0,\ \mbox{as}\ t\rightarrow t_0.
\end{align*}
This proves that  $\lim\limits_{t\rightarrow t_0}\Vert u(t)-u(t_0)\Vert_{M_{2,1}}=0.$ In a similar way  we obtain that 
\begin{align*}
\lim_{t\rightarrow t_0}\Vert v(t)-v(t_0)\Vert_{D^{-1}JM_{2,1}}=0.
\end{align*}
The remain of the proof is standard and for this reason we omit the details.  
\endproof
\subsection {Proof of Theorem \ref{local1}}
We consider the metric space $\mathcal{B}_{2\epsilon}=\{[u,v]\in \mathcal{L}^T_{\alpha,s}:\ \Vert [u,v]\Vert_{\mathcal{L}^T_{\alpha,s}}\leq 2\epsilon\}.$ We prove that the mapping $\Phi[u,v]=[\Phi_1(u,v), \Phi_2(u,v) ]$ defines a contraction in the metric space $\mathcal{B}_{2\epsilon},$ for some $\epsilon>0.$ Let us consider $\epsilon=\tilde{C}(\Vert u_0\Vert{M^s_{p',q}}+ \Vert v_0\Vert{M^s_{p',q}}),$ where $\tilde{C}$ is the maximum of the constants of Lemmas \ref{LinB1B2(1+t)} and \ref{LinB1B3(1+t)}. Take $T>0$ such that $C_3\epsilon^{\lambda-1}2^\lambda T<1,$ where $C_3$ is the constant provided by Proposition \ref{non_estLoc}.
From Lemmas \ref{LinB1B2(1+t)} and \ref{LinB1B3(1+t)}  and Proposition \ref{non_estLoc} we have
\begin{align}
\Vert \Phi[u,v]\Vert_{\mathcal{L}^T_{\alpha, s}}&\leq \sup_{-T<t<T}\vert t\vert^{\alpha}\Vert B_1(t)u_0+B_2(t)v_0\Vert_{M^s_{p,q}}\nonumber\\
&+\sup_{-T<t<T}\vert t\vert^{\alpha}\Vert B_3(t)u_0+B_1(t)v_0\Vert_{D^{-1}JM^s_{p,q}}+C_3T\Vert [u,v]\Vert_{\mathcal{L}^T_{\alpha,s}}^\lambda\nonumber\\
&\leq \tilde{C}(\Vert u_0\Vert{M^s_{p',q}}+ \Vert v_0\Vert{M^s_{p',q}})+C_3T(2\epsilon)^\lambda\nonumber\\
&\leq \epsilon+C_3T(2\epsilon)^\lambda<2\epsilon,\label{ms12}
\end{align}
for all $[u,v]\in \mathcal{B}_{2\epsilon}$ and therefore, $\Phi(\mathcal{B}_{2\epsilon})\subseteq \mathcal{B}_{2\epsilon}.$ On the other hand, from Proposition \ref{non_est}, and applying the Young inequality in each term $\Vert u(t)\Vert^{\lambda-k}_{M^s_{p,q}}\Vert \tilde{u}(t)\Vert^{k-1}_{M^s_{p,q}},$ $k=1,...,\lambda,$ we have
 we have
\begin{align}
\Vert \Phi[u,v]-\Phi[\tilde{u},\tilde{v}]\Vert_{\mathcal{L}^T_{\alpha,s}}&\leq C_3T\Vert [u,v]-[\tilde{u},\tilde{v}]\Vert_{\mathcal{L}^T_{\alpha,s}}(\Vert [u,v]\Vert^{\lambda-1}_{\mathcal{L}^T_{\alpha,s}}+\Vert [\tilde{u},\tilde{v}]\Vert^{\lambda-1}_{\mathcal{L}^T_{\alpha,s}})\nonumber\\
&\leq C_3T2^\lambda\epsilon^{\lambda-1}\Vert [u,v]-[\tilde{u},\tilde{v}]\Vert_{\mathcal{L}^T_{\alpha,s}},\label{ms13}
\end{align}
for all $[u,v], [\tilde{u},\tilde{v}]\in \mathcal{B}_{2\epsilon}.$ Since $C_32^{\lambda}\epsilon^{\lambda-1}T<1,$ from (\ref{ms12}) and (\ref{ms13}) we get that $\Phi$ is a contraction, which implies that the integral equation (\ref{ms5}) has a unique solution $[u,v]\in\mathcal{L}^T_{\alpha,s}$, satisfying $\Vert [u,v]\Vert_{\mathcal{L}^T_{\alpha,s}}\leq{2\epsilon}.$ The proof that the data-solution map is locally Lipschitz follows as in the proof of this property in Theorem \ref{global1}. 
\subsection{Proof of Theorem \ref{teoscat}}
We only prove (\ref{aux-scat1}) in the case $t\rightarrow \infty.$ The case $t\rightarrow -\infty$ follows analogously. We define
\begin{equation*}
[u_0^+,v_0^+]=[u_0,v_0]-\int_0^\infty B(-\tau)[0,f(u(\tau))]d\tau.
\end{equation*}
Let $[u^+,v^+]=B(t)[u_0^+,v_0^+]$ the solution of the linear problem associated to (\ref{EquiSyst}), that is,
\begin{equation*}
\left\{
\begin{array}{lc}
\partial_{t}u^+=\Delta v^+, & \\
\partial_{t}v^+=u^+-\Delta u^+, &  \\
u^+(x,0)=u^+_{0}(x),  &  \\
v^+(x,0)=v^+_0(x).&
\end{array}
\right.
\end{equation*}
The pair $[u^+,v^+]$ can be expressed as
\begin{equation}\label{ms5bi2}
[u^+,v^+]=B(t)[u_0,v_0]-\int_0^\infty B(t-\tau)[0,f(u(\tau))]d\tau.
\end{equation}
Taking the difference between (\ref{ms5}) and (\ref{ms5bi2}) and computing the $M^s_{p,q}\times D^{-1}JM^s_{p,q}$-norm we get
\begin{align*}
&\Vert\lbrack u(t)-u^{+}(t),v(t)-v^{+}(t)]\Vert_{M^s_{p,q}\times D^{-1}JM^s_{p,q}}    =\left \Vert \int_t^\infty B(t-\tau)[0,f(u(\tau))]d\tau\right\Vert_{M^s_{p,q}\times D^{-1}JM^s_{p,q}}\\
&\ \ \ \ \apprle \int_t^\infty (1+\vert t-\tau\vert)^{-\alpha}\Vert u(\tau)\Vert_{M^s_{p,q}}d\tau\\
&\ \ \ \ \apprle \left( \sup_{0<t<\infty}(1+t)^\alpha\Vert u(t)\Vert_{M^s_{p,q}}\right)^{\lambda}\int_t^\infty(1+\vert t-\tau\vert)^{-\alpha}(1+\tau)^{-\alpha\lambda}d\tau\\
&\ \ \ \ \apprle\int_t^\infty(1+\tau)^{-\alpha\lambda}d\tau\\
&\ \ \ \ \leq Ct^{1-\alpha\lambda}.
\end{align*}

\subsection{Proof of Theorem \ref{teoasymp}}
Let $[u,v], [\tilde{u},\tilde{v}]$ be two solutions of
(\ref{GoodBous}) with data $[u_{0},v_{0}],$ $[\tilde{u}_{0},\tilde{v}_{0}]\in M^{s}_{p',q}\times M^{s}_{p',q},$ respectively. We assume only the case $t>0;$ the case $t<0$ can be addressed analogously.  Taking the difference between the integral equations (\ref{ms5}) and computing the $M^s_{p,q}\times D^{-1}JM^s_{p,q}$-norm we get
\begin{align}
&\Vert\lbrack u(t)-\tilde{u}(t),v(t)-\tilde{v}(t)]\Vert_{M^s_{p,q}\times D^{-1}JM^s_{p,q}} \leq \left \Vert B(t)[u_0-\tilde{u}_0,v_0-\tilde{v}_0]\right\Vert_{M^s_{p,q}\times D^{-1}JM^s_{p,q}}\nonumber\\
&\ \ \ +\left \Vert \int_0^t B(t-\tau)[0,f(u)-f(\tilde{u})]d\tau\right\Vert_{M^s_{p,q}\times D^{-1}JM^s_{p,q}}:=J_1+J_2.\label{asy89}
\end{align}
Working as in Proposition \ref{non_est}, and using that $[u,v], [\tilde{u},\tilde{v}]\in B_{2\epsilon}\subset {\mathcal{L}^\infty_{\alpha,s}},$ we bound $J_2$ as follows
\begin{align}
J_2 &\apprle\int_0^t(1+ \vert t-\tau\vert)^{-\alpha}\Vert u-\tilde{u}\Vert_{M^s_{p,q}}\sum_{k=1}^{\lambda}\Vert u\Vert^{\lambda-k}_{M^s_{p,q}}\Vert \tilde{u}\Vert^{k-1}_{M^s_{p,q}}d\tau\nonumber\\
&\apprle 2^\lambda\epsilon^{\lambda-1}\int_0^t(1+ \vert t-\tau\vert)^{-\alpha}(1+ \tau)^{-\alpha\lambda+1}\Vert\lbrack u(t)-\tilde{u}(t),v(t)-\tilde{v}(t)]\Vert_{M^s_{p,q}\times D^{-1}JM^s_{p,q}} d\tau.\label{2asymp2}
\end{align}
Let us denote by 
\begin{align*}
H:=\limsup_{t\rightarrow \infty}(1+t)^{\alpha}\Vert [u(t)-\tilde{u}(t),v(t)-\tilde{v}(t)]\Vert_{M^s_{p,q}\times D^{-1}JM^s_{p,q}}.
\end{align*}
From Theorem \ref{global1}, $H<\infty.$ Moreover, from (\ref{asy89}) and (\ref{2asymp2}) we obtain
\begin{align}
H&\apprle (1+t)^{\alpha}J_1+\left( 2^\lambda \epsilon^{\lambda-1}(1+t)^{\alpha}\int_0^t(1+ \vert t-\tau\vert)^{-\alpha}(1+ \tau)^{-\alpha\lambda}d\tau\right)H\nonumber\\
 &\apprle (1+t)^{\alpha}J_1+\left( 2^\lambda \epsilon^{\lambda-1}\right)H.\label{ass99}
\end{align}
Thus, if $\lim_{t\rightarrow \infty}\Vert B(t)[u_0-\tilde{u}_0,v_0-\tilde{v}_0]\Vert_{M^s_{p,q}\times D^{-1}JM^s_{p,q}}=0,$ and remembering that $2^\lambda \epsilon^{\lambda-1}<1,$ from (\ref{ass99}) we get that $H=0,$ and then
$$\lim_{t\rightarrow \infty}(1+t)^{\alpha}\Vert [u(t)-\tilde{u}(t),v(t)-\tilde{v}(t)]\Vert_{M^s_{p,q}\times D^{-1}JM^s_{p,q}}=0.$$
Now, we will prove the converse proposition. First, notice that from Theorem \ref{global1}
$$\left(\Vert [u,v]\Vert_{\mathcal{L}^\infty_{\alpha,s}}\right)^{\lambda-1}+\left(\Vert [\tilde{u},\tilde{v}]\Vert_{\mathcal{L}^\infty_{\alpha,s}}\right)^{\lambda-1}=K<\infty.$$
Then we get,
\begin{align*}
&\limsup_{t\rightarrow \infty}(1+t)^\alpha\Vert B(t)[u_0-\tilde{u}_0,v_0-\tilde{v}_0\Vert_{M^s_{p,q}\times D^{-1}JM^s_{p,q}}\\
&\ \ \ \leq \limsup_{t\rightarrow \infty}(1+t)^\alpha\Vert [u(t)-\tilde{u}(t),v(t)-\tilde{v}(t)\Vert_{M^s_{p,q}\times D^{-1}JM^s_{p,q}}\\
&\ \ \ \ + \limsup_{t\rightarrow \infty}(1+t)^\alpha\left \Vert \int_0^t B(t-\tau)[0,f(u)-f(\tilde{u})]d\tau\right\Vert_{M^s_{p,q}\times D^{-1}JM^s_{p,q}}\\
&\ \ \ \ \leq 0+CKH=0.
\end{align*}
Thus we conclude the proof of Theorem \ref{teoasymp}.
\endproof

\end{document}